\documentclass[reqno,12pt,oneside]{amsart}
\usepackage{amsthm}
\usepackage{mathrsfs,hyperref} 

\usepackage[dvipsnames]{xcolor}

\def\C{\mathbb{ C}}

\def\L{\mathscr{ L}}
\def\N{\mathbb{ N}}

\def\Q{\mathbb{Q}}

\def\epsi{\mathcal{\varepsilon}}

\newcommand{\half}{\tfrac12}
\newcommand{\ab}[2]{\genfrac{}{}{0pt}{}{#1}{#2}}

\newtheorem{thm}{Theorem}
\newtheorem{lem}[thm]{Lemma}
\newtheorem{prop}[thm]{Proposition}
\newtheorem{cor}[thm]{Corollary}

\newtheorem{remark}{Remark}

\numberwithin{equation}{section}

\newcommand{\pf}{{\bf Proof: \/}}
\def\qed{\nobreak\quad\raise -2pt\hbox{\bf\vrule\vbox to 10pt{\hrule width 6pt \vfill\hrule}\vrule}\par\vspace{2ex}}

\title
[Moments of Orthogonal Polynomials]{Moments of Orthogonal Polynomials and Exponential Generating Functions}
\author{Ira M. Gessel}
\address[Ira M. Gessel]{Department of Mathematics, Brandeis University, Waltham, MA 02453, USA}
\email{gessel@brandeis.edu}
\author{Jiang Zeng}
\address[Jiang Zeng]{Universit\'{e} de Lyon; Universit\'{e} Lyon 1; Institut Camille Jordan; UMR 5208 du CNRS; 43, boulevard du 11 novembre 1918, F-69622 Villeurbanne Cedex, France}
\email{zeng@math.univ-lyon1.fr}
\dedicatory{ Dedicated to the memory of Richard Askey}

\begin{document}

\begin{abstract} Starting from the moment sequences of classical orthogonal polynomials we derive the orthogonality purely algebraically. We consider also the moments of ($q=1$) classical
orthogonal polynomials, and study those cases in which the exponential generating function has a nice form. 
In the opposite direction,  we show that the  generalized Dumont-Foata polynomials with six parameters 
 are the moments of  rescaled 
continuous dual Hahn polynomials. Finally we  show that 
one of our 
methods   can be applied to deal with the moments of  Askey-Wilson polynomials.
\end{abstract}

\maketitle
\section{Introduction}

Many of the most important sequences in enumerative combinatorics---the
factorials, derangement numbers, Bell numbers, Stirling polynomials, secant numbers, tangent
numbers, Eulerian polynomials, Bernoulli numbers, and Catalan
numbers---arise as moments of well-known orthogonal polynomials.
With the exception of the Bernoulli and Catalan numbers,
these orthogonal polynomials are all Sheffer type; see \cite{Vi83, Ze92}.
One characteristic of these sequences is that their ordinary
generating functions have simple continued fractions.
For some recent work on the moments of classical orthogonal polynomials we refer the reader to 
\cite{Cha, Dom, NKF, CKS, CSSW, KS}.

There is another sequence which appears in a number of enumerative
applications, and which also has a simple continued fraction.
The {\sl Genocchi numbers} may be defined by
$$
\sum_{n=0}^{\infty}G_n\frac{x^n}{n!}
=\sum_{n=0}^{\infty}G_{2n+2}\frac{x^{2n+2}}{(2n+2)!}
=x \tan\frac x2.
$$
So $G_n=0$ when $n$
is odd or $n=0$, $G_2=1$, $G_4=1$, $G_6=3$, $G_8=17$, $G_{10}=155$,
and $G_{12}=2073$.

A closely related sequence is the median Genocchi numbers $H_{2n+1}$, which first appeared in Seidel's work \cite{Sei}.
These numbers 
 do not seem to have a simple exponential generating function
 and may be defined by $H_1=1$, and for $n\geq 1$,
\begin{equation}
H_{2n+1}=\sum_{k=1}^{\lfloor (n+1)/2\rfloor}(-1)^{k-1} {n\choose 2k-1} G_{2n+2-2k}.
\end{equation}
So $H_1=1$, $H_3=1$, $H_5=2$, $H_7=8$, $H_7=56$, $H_{11}=608$, 
$H_{13}=9440$; see \cite{Vi81, DZ, Fe, Ge}.

A comprehensive discussion of the combinatorial properties of Genocchi
numbers has been given by Viennot \cite{Vi81}.
In particular, he showed  that the Genocchi numbers and median Genocchi numbers $H_{2n+1}$ have the S-fraction expansions
\begin{align}
\sum_{n=0}^{\infty}G_{2n+2}\,t^{2n}&=S(t^2; 1^2, 1\cdot 2, \, 2^2,\, 2\cdot 3,\, 3^2,\, 3\cdot 4,\, 4^2,\,\ldots)\notag\\
  &=\cfrac{1}{1-
  \cfrac{1^2\,t^2}{1-
  \cfrac{1\cdot 2\, t^2} {1-
  \cfrac{2^2\, t^2} {1-
  \cdots
  }}}}
\label{genocchi}\\
\sum_{n=0}^{\infty} H_{2n+1}\,t^{2n}&=S(t^2; \, 1^2,\, 1^2,\, 2^2,\,2^2,\,3^2,\, 3^2,\,\ldots)\notag\\
  &=\cfrac{1}{1-
  \cfrac{1^2\,t^2}{1-
  \cfrac{1^2\, t^2} {1-
  \cfrac{2^2\, t^2} {1-
  \cdots
}}}}
\label{mgenocchi}
\end{align}
Some recent papers (see \cite{Bi,He, LW}) have 
shown  there are renewed interest on Genocchi  numbers and median Genocchi numbers.

In his combinatorial approach to orthogonal polynomials \cite[p.~V-10]{Vi83},
Viennot briefly alludes to the monic orthogonal polynomials 
whose moments are the Genocchi
numbers (i.e., the $n$th moment is $G_{2n+2}$) and  
the median Genocchi numbers $H_{2n+1}$
but he does not give any explicit formula for them.

Many years ago, one of the authors (IG) was learning about continued fractions and their connection to orthogonal polynomials, and saw in  Viennot's work \cite{Vi83} the simple continued fraction for the Genocchi numbers. He wondered what the corresponding  orthogonal polynomials were and wrote to Richard Askey, asking if he knew. Askey immediately wrote back to say that they were a special case of the continuous dual Hahn polynomials, and this paper grew out of an attempt to understand his reply.

As the Wilson polynomials are the most general ($q=1$) classical orthogonal polynomials, we shall first consider  the moments of the Wilson polynomials. In the general case, there are four parameters $a,b,c,d$ and 
the ordinary  generating function of these moments
can be expressed as a hypergeometric series.  We show that these moments have a simple exponential generating function when $a=0$ or  $a=1/2$.

We then consider the continuous dual Hahn polynomials and the (continuous) Hahn polynomials along with their rescaled versions. In particular, we show that the moments of
a rescaled version of the continuous dual Hahn polynomials  are  the generalized Dumont-Foata polynomials, which are a refinement of both Genocchi numbers and median Genocchi numbers. It would be possible to derive the moment generating function for Hahn polynomials from that for Wilson polynomials, but we will give a separate derivation and use our method for variety. Moreover, we show 
show that the latter  method can be applied to derive immediately 
a formula for  the moments of 
the Askey-Wilson polynomials.

\section{Moments of Wilson polynomials}
The monic Wilson polynomials  $W_n(x)$ (see \cite{AAR, Wi,KLS})
are defined by
\begin{equation}\label{wilson}
W_n(x^2)=(-1)^n \frac{(a+b)_n (a+c)_n(a+d)_n}{(a+b+c+d+n-1)_n}\widetilde{W}(x^2;a,b,c,d) \end{equation}
where
$$
\widetilde{W}(x^2;a,b,c,d)={}_4F_3\left(\begin{array}{c}
-n, \,n+a+b+c+d-1, \,a+ix,\,a-ix\cr
a+b,\,a+c, \,a+d
\end{array};1\right).
$$
If $a,b,c,d$ are positive or 
$a=\bar b$ and/or $c=\bar d$ and the real parts are positive, 
 then the orthogonality  reads as follows: 
\begin{subequations}
\begin{multline}
\label{wilson-ortho}
\quad \frac{1}{2\pi}\int_0^\infty\left|
\frac{\Gamma(a+ix)\Gamma(b+ix)\Gamma(c+ix)\Gamma(d+ix)}
{\Gamma(2ix)}\right|^2W_m(x^2)W_n(x^2)\,dx\\
=
\frac{n!\,\Gamma(n+a+b)\cdots \Gamma(n+c+d)}
 {(n+a+b+c+d-1)_n\,\Gamma(2n+a+b+c+d)}\,\delta_{mn}.
\quad
\end{multline}
Consider the moment sequence
$w_n(a):=w_n(a,b,c,d)$ of the Wilson polynomials defined by
\begin{multline}\label{wmoment}
\quad\frac{1}{2\pi}\int_0^\infty\left|\frac{\Gamma(a+ix)
\Gamma(b+ix)\Gamma(c+ix)\Gamma(d+ix)}
{\Gamma(2ix)}\right|^2x^{2n}dx\\
=w_{n}(a){\Gamma(a+b)\Gamma(a+c)\Gamma(a+d)
\Gamma(b+c)\Gamma(b+d)\Gamma(c+d)\over \Gamma(a+b+c+d)}.
\quad
\end{multline}
\end{subequations}
It follows that $w_0(a)=1$ and for $n\geq 0$,
\begin{equation}
w_{n+1}(a)=
{(a+b)(a+c)(a+d)\over a+b+c+d}w_n(a+1)-a^2 w_n(a).\label{eq:moment}
\end{equation}

\begin{prop} There holds
\begin{align*}
\sum_{n\geq 0}w_n(a)t^n&=\sum_{n\geq 0}{(a+b)_n(a+c)_n(a+d)_nt^{n}\over
(a+b+c+d)_n\prod_{l=0}^n(1+(a+l)^2t)}\\
  &={1\over 1+a^2t}\;
{}_4F_3\left(\begin{array}{c}
a+b, \,a+c, \,a+d,\,1\cr
a+b+c+d,\,a+1+{i/\sqrt{t}}, \,a+1-{i/\sqrt{t}}
\end{array};1\right).
\end{align*}
\end{prop}
\begin{proof} Set $F(a,t)=\sum_{n\geq 0}w_n(a)t^n$. The recurrence
(\ref{eq:moment}) implies that
\begin{equation}
F(a,t)= 1+\frac{(a+b)(a+c)(a+d)}{(a+b+c+d)}tF(a+1,t)-a^2tF(a,t).
\end{equation}
Hence
$$
F(a,t)={1\over 1+a^2t}+\frac{(a+b)(a+c)(a+d)t}{(a+b+c+d)(1+a^2t)}F(a+1,t).
$$
The formula follows then by iterating the above functional equation.
\end{proof}
Here is the formal  approach to the Wilson polynomials. 
We define a linear functional $\L$  on even polynomials by $\L(x^{2n})=w_n(a)$. 
\begin{lem}\label{lem1}
For $k\geq 0$,
\begin{equation}
\L(x^{2n}(a+ix)_k\, (a-ix)_k)=
\frac{(a+b)_k\,(a+c)_k\,(a+d)_k}{(a+b+c+d)_k}\,w_n(a+k).
\end{equation}
\end{lem}
\begin{proof}
We prove this by induction on $k$. It is clear for $k=0$. 
For  $k\geq 0$, we have 
\begin{align*}
\L(x^{2n}(a+ix)_{k+1}\, &(a-ix)_{k+1})\\
&=\L\left(x^{2n}(a+ix)_{k}\, (a-ix)_k ((a+k)^2+x^2)\right)
 \\
&=\frac{(a+b)_k\,(a+c)_k\,(a+d)_k}{(a+b+c+d)_k}
\left((a+k)^2\,w_n(a+k) +w_{n+1}(a+k)\right).
\end{align*}
By (\ref{eq:moment}) the formula  is valid for $k+1$.
\end{proof}

The Wilson polynomials  have 
the  orthogonality relation \cite[(9.1.2)]{KLS}
\begin{multline}\label{orth-wilson}
\L(W_n(x^2)W_m(x^2))\\
=n!
\frac{(a+b)_n(a+c)_n(a+d)_n(b+c)_n(b+d)_n(c+d)_n}{(a+b+c+d+n-1)_n
(a+b+c+d)_{2n}}\delta_{mn}\qquad (m, n\geq 0).
\end{multline}
We can verify the orthogonality directly using 
Lemma~\ref{lem1}. Indeed, by induction on $n\geq 0$,  
\begin{multline}\label{mnMoments}
\qquad
\L\left((a+ix)_m\,(a-ix)_m\, (b+ix)_n\, (b-ix)_n\right)\\
 =
\frac{(a+b)_{m+n}(a+c)_m(a+d)_m(b+c)_n(b+d)_n}{(a+b+c+d)_{m+n}}.
\qquad
\end{multline}
We  evaluate
\begin{align*}
\L\kern15pt&\kern-15pt\left(\tilde{W}_n(x^2;a,b,c,d)\, (b+ix)_m\, (b-ix)_m\right)\\
&=\frac{(a+b)_m(b+c)_m(b+d)_m}{(a+b+c+d)_m}{}_3F_2\left(\begin{array}{c}
-n, \,n+a+b+c+d-1, \,a+b+m\cr
a+b,\,a+b+c+d+m
\end{array};1\right)\\
&=\frac{(a+b)_m(b+c)_m(b+d)_m}{(a+b+c+d)_m}\frac{(1-n-c-d)_n(-m)_n}{(a+b)_n(
1-n-a-b-c-d-m)_n}\\ 
&\qquad \hbox{(by the Pfaff-Saalsch\"utz theorem)}\\
&=\frac{(a+b)_m(b+c)_m(b+d)_m(c+d)_n(-m)_n}{(a+b+c+d)_{m+n}(a+b)_n}.
\end{align*}
This is 0 for $m<n$, which implies orthogonality.

The orthogonality \eqref{orth-wilson}  implies 
 the three-term recurrence
relation~\cite[(9.1.5)]{KLS}
\begin{equation}\label{recurrence-favard}
xW_n(x)=W_{n+1}(x)+b_n W_n(x)+\lambda_n W_{n-1}(x)
\end{equation}
with $
\lambda_n=A_{n-1}C_n,\quad b_n=A_n+C_n-a^2$,
where
\begin{equation*}
\left\lbrace\begin{aligned}
A_n&=\frac{\displaystyle(n+a+b+c+d-1)(n+a+b)(n+a+c)(n+a+d)}
{\displaystyle(2n+a+b+c+d-1)(2n+a+b+c+d)},\\[5pt]
C_n&=\frac{\displaystyle n(n+b+c-1)(n+b+d-1)(n+c+d-1)}
{\displaystyle(2n+a+b+c+d-2)(2n+a+b+c+d-1)}.       
\end{aligned}\right.
\end{equation*}
Indeed, from \eqref{orth-wilson} we derive that 
$$
\lambda_n= \frac{\L(x^2W_{n-1}(x^2)W_{n}(x^2))}
{\L(W_{n-1}(x^2)^2)}=\frac{\L(W_{n}(x^2)^2)}
{\L(W_{n-1}(x^2)^2)}=A_{n-1}C_n.
$$
Extracting the coefficient of $x^n$ in  \eqref{recurrence-favard}  we have
\begin{equation}\label{coeff:bn}
b_n=[x^{n-1}]W_n(x)-[x^n]W_{n+1}(x)
\end{equation}
where $[x^k]W_n(x)$ is the coefficient of $x^k$ in $W_n(x)$.
As $$(a+i\sqrt{x})_k(a-i\sqrt{x})_k=\prod_{l=0}^{k-1}((a+l)^2+x),$$
 we derive  from \eqref{wilson} that 
\[
[x^{n-1}]W_n(x)
=-\frac{n(a+b+n-1)(a+c+n-1)(a+d+n-1)}{a+b+c+d+2n-2}	
  +\sum_{l=0}^{n-1}(a+l)^2, 
\]
which yields $b_n=A_n+C_n-a^2$ by \eqref{coeff:bn}.

It is known \cite{Ch} that the recurrence 
\eqref{recurrence-favard}  is equivalent to the following J-fraction expansion of  the  moments
$w_n(a)$, where the J-fraction 
$J(t;a_1,b_1,a_2,b_2,\dots) $ is defined to be 
\begin{equation*}
\cfrac{1}{1-a_1 t-
\cfrac{b_1 t^2}{1-a_2 t -
\cfrac{b_2 t^2}{1-\cdots
}}}.
\end{equation*}

\begin{prop} 
We have 
expansion
\begin{align}\label{W-JCF}
\smash[b]{\sum_{n=0}^\infty} w_n(a)t^n&= J(t;\, A_0+C_0-a^2, A_0C_1,\,\ldots,\nonumber\\ 
& \qquad\qquad A_n+C_n-a^2,\, A_{n}C_{n+1}, \ldots).
\end{align}
\end{prop}

Recall  the  following  contraction formulae \cite{Vi83} transforming S-fraction to  J-fraction:
\begin{subequations}
\begin{align}
S(t; \alpha_1, \ldots, \alpha_n,\ldots)
&=J(t; \gamma_0, \beta_1, \gamma_0, \beta_1, \ldots, 
\gamma_n, \beta_{n+1},\ldots)\label{contraction1}\\
&=1+\gamma_0 t\, 
J(t;\gamma_0', \beta_1', \ldots, \gamma_n', \beta'_n, \ldots )
\label{contraction2}
\end{align}
with $\gamma_0=\alpha_1$,$\gamma'_0=\alpha_1+\alpha_2$ and for $n\geq 1$
\begin{align*}
 \gamma_n&=\alpha_{2n}+\alpha_{2n+1},\qquad 
\beta_n=\alpha_{2n-1}\alpha_{2n};\\
\gamma'_n&=\alpha_{2n-1}+\alpha_{2n},\qquad 
\beta'_n=\alpha_{2n}\alpha_{2n+1}.
\end{align*}
\end{subequations}

Thus, when $a=0$ we can transform \eqref{W-JCF} to 
 the  S-fraction.
 \begin{cor} We have 
\begin{multline}\label{w:a=0}
\sum_{n=0}^\infty w_n(0)t^n= S\left(t;\, \frac{bcd}{b+c+d},\,
\frac{(b+c)(c+d)(b+d)}{(b+c+d)(b+c+d+1)},\,\cdots,\,\right.\\
 \frac{(b+n)(c+n)(d+n)(b+c+d+n-1)}{(b+c+d+2n-1)(b+c+d+2n)},\\
\left.\frac{(n+1)(b+c+n)(c+d+n)(b+d+n)}{(b+c+d+2n)(b+c+d+2n+1)},
\ldots\right).\qquad
\end{multline}
\end{cor}


\section{Exponential generating functions}

To derive exponential generating functions from ordinary generating functions we use the following lemma, which we will also apply to other orthgonal polynomials.

Let $\epsi: \Q[t]\longrightarrow \Q[t]$ be the linear
transformation defined by
\begin{equation}\label{eq:epsilon}
\epsi\biggl(\sum_{n=0}^\infty u_nt^n\biggr)=\sum_{n=0}^\infty u_n\frac{t^n}{n!}.
\end{equation}
\begin{lem} 
\label{l-egf}
For any nonnegative integers $m$ and $n$ we have
$$
\epsi\left(\frac{t^m}{(1-\alpha t)(1-(\alpha+1)t)\cdots (1-(\alpha+m)t)}\right)=e^{\alpha t}\frac{(e^t-1)^m}{m!}. $$
\end{lem}
\begin{remark}
We shall give two proofs. The first one uses a partial fraction expansion, while the second one  doesn't require explicitly doing the partial fraction expansion.
\end{remark}
\begin{proof}[First proof]
Expanding the left side by partial fractions
 we get
$$
\frac{t^m}{\prod_{k=0}^m(1-(\alpha+k) t)}=
\frac{1}{m!}\sum_{k=0}^m\binom{m}{k}
\frac{(-1)^k}{1-(\alpha+m-k)\,t}.
$$
The result follows  from the fact that
$\epsi(\frac{1}{1-\beta \,t})=e^{\beta\,t}$ and the binomial theorem.
\end{proof}

%
\begin{proof}[Second proof]
Expanding $e^{\alpha t}{(e^t-1)^m}/{m!}$ by the binomial theorem we obtain a linear combination of terms
$e^{(\alpha+k)t}$, $k=0,\dots, m$, and $\epsilon^{-1}(e^{(\alpha+k)t})= 1/(1-(\alpha+k)t)$. So $\epsilon^{-1}\bigl(e^{\alpha t}{(e^t-1)^m}/{m!}\bigr)$ must be a proper rational function
of the form 
\begin{equation*}
\frac{N(t)}{(1-\alpha t)(1-(\alpha+1)t)\cdots (1-(\alpha+m)t)},
\end{equation*}
where the degree of $N(t)$ is at most $m$. But since the first nonzero coefficient of $e^{\alpha t}(e^t-1)^m/m!$ is $t^m/m!$, $N(t)$ must be $t^m$.
\end{proof}

\begin{remark}
We may also reduce the first proof to  $\alpha=0$ and apply the  formula
\begin{equation}\label{alpha-lem}
\epsilon^{-1}\left(
e^{\alpha t}\sum_{n=0}^\infty u_n\frac{t^n}{n!}\right)=\frac{1}{1-\alpha t}\sum_{n=0}^\infty u_n\left(\frac{t}{1-\alpha t}\right)^n.
\end{equation}

\end{remark}
\begin{subequations}
If $\alpha=-m/2$ then we have
\begin{equation}
\epsi
\left(\frac{t^m}{(1+\frac{m}{2} t)\cdots (1-\frac{m}{2}t)}\right)
=e^{-\frac{m}{2} t}\frac{(e^t-1)^m}{m!}=\frac{(2\sinh\frac{t}{2})^m}{m!}.
\end{equation}
If $m=2n$, this is
\begin{equation}
\label{m=2n}
\epsi\left(\frac{t^{2n}}{(1- t^2)\,(1-2^2\, t^2)\cdots (1-n^2\, t^2)}\right)=\frac{(2\sinh\frac{t}{2})^{2n}}{(2n)!}, \end{equation}
and if $m=2n+1$ this is
\begin{equation}
\label{m=2n+1}
\epsi\left(\frac{t^{2n+1}}{(1- (\frac{1}{2})^2\,t^2)\,(1-(\frac{3}{2})^2\, t^2)\cdots (1-(n+\frac{1}{2})^2\, t^2)}\right) =\frac{(2\sinh\frac{t}{2})^{2n+1}}{(2n+1)!}.
\end{equation}
\end{subequations}
If $a=0$ or $a={1\over 2}$ then there is a
nice exponential generating function for the moments of the Wilson polynomials.

\begin{thm} We have 
\begin{equation}\label{gfwzero}
\sum_{n=0}^\infty w_n(0)\frac{t^{2n}}{(2n)!}={}_3F_2\left(\begin{array}{c}
b, \,c, \,d\cr
b+c+d,\,\frac{1}{2}
\end{array};\sin^2\frac{t}{2}\right),
\end{equation}
\begin{equation}\label{gfwhalf}
\sum_{n=0}^\infty w_n(1/2)\frac{t^{2n+1}}{(2n+1)!}
=2\sin\frac{t}{2}\:{}_3F_2\left(
\ab{b+\half, \,c+\half, \,d+\half}
{b+c+d+\half,\,\frac{3}{2}} \,
;\sin^2\frac{t}{2}\right).
\end{equation}
\end{thm}
\begin{proof}
Returning to the moments of the Wilson polynomials, 
for $a=0$ we have
$$
\sum_{n=0}^\infty (-1)^nw_n(0)\, t^{2n}=\sum_{n=0}^\infty (-1)^n
\frac{(b)_n(c)_n(d)_n\, t^{2n}}{(b+c+d)_n\prod_{l=0}^n(1-l^2t^2)}.
$$
So by \eqref{m=2n},
\begin{align*}
\sum_{n=0}^\infty (-1)^nw_n(0)\, \frac{t^{2n}}{(2n)!}
&=\sum_{n=0}^\infty (-1)^n
\frac{(b)_n(c)_n(d)_n}{(b+c+d)_n}\frac{(2\sinh \frac{t}{2})^{2n}}{(2n)!}\\
&= {}_3F_2\left(\begin{array}{c}
b, \,c, \,d\cr
b+c+d,\,\frac{1}{2}
\end{array};-\sinh^2\frac{t}{2}\right).
\end{align*}
Replacing $t$ with $it$, we get \eqref{gfwzero}.

Similarly, for $a=1/2$ we have
$$
\sum_{n=0}^\infty w_n(1/2)t^n=\sum_{n=0}^\infty \frac{(b+1/2)_n(c+1/2)_n(d+1/2)_n(-1)^nt^{2n}}{(b+c+d+1/2)_n\prod_{l=0}^n(1-
(l+1/2)^2t^{2n})}. $$
Hence
$$
\sum_{n=0}^\infty (-1)^nw_n(1/2)t^{2n+1}=
\sum_{n=0}^\infty (-1)^n\frac{(b+1/2)_n(c+1/2)_n(d+1/2)_nt^{2n+1}}
{(b+c+d+1/2)_n\prod_{l=0}^n(1-(l+1/2)^2t^{2n})}.
$$
Therefore by \eqref{m=2n+1},
$$
\sum_{n=0}^\infty w_n(1/2)\frac{(-1)^nt^{2n+1}}{(2n+1)!}
=\sum_{n=0}^\infty (-1)^n\frac{(b+1/2)_n(c+1/2)_n(d+1/2)_n}
{(b+c+d+1/2)_n}\frac{(2\sinh \frac{t}{2})^{2n+1}}{(2n+1)!}.
$$
Replacing $t$ with $it$, we may write this as \eqref{gfwhalf}.
\end{proof}

We can also get the exponential generating functions
\eqref{gfwzero} and \eqref{gfwhalf} as follows.
We have
$$
\cosh\left(2x\arcsin\frac{z}{2}\right)=\sum_{n=0}^\infty x^2(x^2+1^2)\cdots (x^2+(n-1)^2)\frac{z^{2n}}{(2n)!}. $$
So, by Lemma~\ref{lem1} with $a=0$, we have
$$
\L\left(\cosh\left(2x\arcsin\frac{z}{2}\right)\right)
=\sum_{n=0}^\infty\frac{(b)_n(c)_n(d)_n}{(b+c+d)_n}\frac{z^{2n}}{(2n)!}.
$$
Setting $z=2\sin\frac{t}{2}$, we get
\begin{align}
\L\left(\cosh(xt)\right)=& \sum_{n=0}^\infty \L(x^{2n})
\frac{t^{2n}}{(2n)!}\nonumber\\
=& {}_3F_2\left(\begin{array}{c}
b, \,c, \,d\cr
b+c+d,\,\frac{1}{2}
\end{array};\;\sin^2\frac{t}{2}\right).
\end{align}
Similarly,
\begin{multline*}
\quad\sinh\left(2x\arcsin\frac{z}{2}\right)\\
=x\sum_{n=0}^\infty
\left(x^2+\left(\tfrac{1}{2}\right)^2\right)\,\left(x^2+\left(\tfrac{3}{2}\right)^2\right)\cdots
\left(x^2+\left(n-\tfrac{1}{2}\right)^2\right)\frac{z^{2n+1}}{(2n+1)!}.\quad
 \end{multline*}
So, by Lemma~\ref{lem1} with $a=\frac{1}{2}$, we have
\begin{align}
\L\left(x^{-1}\sinh(xt)\right)=& \sum_{n=0}^\infty \L(x^{2n})
\frac{t^{2n+1}}{(2n+1)!}\nonumber\\
=& 2\sin\frac{t}{2}\,{}_3F_2\left(\begin{array}{c}
b+\half, \,c+\half, \,d+\half\cr
b+c+d,\,\frac{1}{2}
\end{array};\;\sin^2\frac{t}{2}\right).
\end{align}

\section{Moments of continuous dual Hahn polynomials}
If we take the limit as $d\to \infty$ in the Wilson polynomials we get the
\emph{continuous dual Hahn polynomials} $p_n(x):=p_n(x;a,b,c)$ defined by
\begin{equation}\label{hahn}
p_n(x^2)=(-1)^n(a+b)_n(a+c)_n\,{}_3F_2\left(\begin{array}{rl}
-n,a+ix,&a-ix\cr
a+b,&a+c
\end{array};1\right).
\end{equation}
The first two values of $p_n(x^2;a,b,c)$  are the following:
\begin{align*}
p_1(x^2)&= x^2-(ab+bc+ca)\\
p_2(x^2)&= x^4-[1+2(a+b+c)+2(ab+ac+bc)]x^2\\       
&+ {a}^{2}{b}^{2}+2\,{a}^{2}bc+{a}^{2}{c}^{2}+2\,{b}^{2}ac+2\,{c}^{2}ab+{
b}^{2}{c}^{2}\\
&+{a}^{2}b+{a}^{2}c+{b}^{2}a+4\,abc+a{c}^{2}+{b}^{2}c+b{c}
^{2}+ab+ac+bc.
\end{align*}
When  either $a$, $b$, and $c$ are all positive or one is positive and the other two are complex conjugates with positive real parts, 
Wilson's result \cite{Wi}  reduces  to
\begin{multline}
\qquad\frac{1}{2\pi}\int_0^\infty\left|\frac{\Gamma(a+ix)\Gamma(b+ix)\Gamma(c+ix)
} {\Gamma(2ix)}\right|^2p_m(x^2)p_n(x^2)\,dx\\
\hskip 1cm =\Gamma(n+a+b)\Gamma(n+a+c)\Gamma(n+b+c)n!\,\delta_{mn}.\qquad
\end{multline}
The corresponding moments  $\mu_n(a,b,c)$ are given by  
\begin{multline}\label{Hmoment}
\qquad\frac{1}{2\pi}\int_0^\infty\left|\frac{\Gamma(a+ix)\Gamma(b+ix)\Gamma(c+ix)
} {\Gamma(2ix)}\right|^2x^{2n}dx\\
 =\mu_{n}(a,b,c)\Gamma(a+b)\Gamma(a+c)\Gamma(b+c).\qquad
\end{multline}
De Branges \cite{DB} (\cite[p.~152]{AAR} and \cite{As1})
also proved that $\mu_0(a,b,c)=1$. The counterpart of \eqref{eq:moment} reads as follows:
\begin{equation}\label{df}
\mu_{n+1}(a,b,c)=(a+b)(a+c)\mu_n(a+1,b,c)-a^2\mu_n(a,b,c), 
\end{equation}
which is equivalent to the generating function  
\begin{align}\label{gf-F}
\sum_{n\geq 0}\mu_{n}(a,b,c)t^n&=\sum_{n\geq 0}\frac{(a+b)_{n}(a+c)_{n}t^n}{\prod_{l=0}^{n}(1+(a+l)^2t)}\nonumber\\
 &=\frac{1}{ 1+a^2t}\;
{}_3F_2\left(\begin{array}{c}
a+b, \,a+c, \,1\cr
a+1+{i/\sqrt{t}}, \,a+1-{i/\sqrt{t}}
\end{array};1\right).
\end{align}
We remark that the above recurrence is nothing but the definition
of a sequence of polynomials introduced by Dumont and Foata \cite{DF,Ca2,Ze}
as an extension of Genocchi numbers. In this context the following result was conjectured by Gandhi \cite{Gan} in 1970 and first
proved by Carlitz \cite{Ca1}, and Riordan and Stein \cite{RS}.
Here we provide a direct proof starting from (\ref{Hmoment}).

\begin{prop} For $n\ge0$ we have 
\begin{equation}\label{gandhi}
\mu_{n}(1,1,1)=G_{2n+4}.
\end{equation}
\end{prop}
\pf
Since $|\Gamma(ix)|^2=\frac{\pi}{x\sinh({\pi}x)}$ and $|\Gamma(1+ix)|^2=\frac{\pi x}{\sinh(\pi x)}$, we have
\begin{align*}
\mu_n(1,1,1)&=\frac{1}{2\pi}\int_0^\infty x^{2n}
\left|\frac{(\Gamma(1+ix))^3}{\Gamma(2ix)}\right|^2dx\cr
&=2\pi\int_0^\infty x^{2n+4} \frac{\cosh ({\pi}x)}{\sinh^2({\pi}x)}dx\cr
&= -2\int_0^\infty x^{2n+4} d(1/(\sinh({\pi}x))).
\end{align*}
Integrating by parts yields
$$
\mu_n(1,1,1)= 4(n+2)\int_0^\infty \frac{x^{2n+3} dx}{\sinh({\pi}x)}.
$$
Equation (\ref{gandhi}) follows then from the known
integral expression of Bernoulli numbers~\cite[p.~39]{Er}:
$$
|B_{2n}|=\frac{2n}{2^{2n}-1}\int_0^\infty \frac{x^{2n-1} dx}{\sinh({\pi}x)},
$$
and the formula $G_{2n}=2(2^{2n}-1)|B_{2n}|$. \qed \medskip

From the J-fraction \eqref{W-JCF}  for the moments of 
Wilson polynomials we derive 
the following J-fraction
\begin{align}
\sum_{n\geq 0}\mu_{n}(a, b, c)t^n
&=J(t; (ab+bc+ca),\, (a+b)(b+c)(c+a),\dots\nonumber\\
&\qquad (a+n)(b+n)+(b+n)(c+n)+(c+n)(a+n)-n(n+1),\nonumber\\
 &\qquad (n+1)(a+b+n)(b+c+n)(c+a+n),\ldots).
\end{align}

By \eqref{w:a=0}, when $a=0$ we have  the S-fraction 
\begin{align}
\sum_{n=0}^\infty \mu_{n}(0, b,c) t^n=
S(t;\, bc,\, b+c,\, (b+1)(c+1), \, 2(b+c+1),\ldots ),
\end{align}
as is well-known, and $b=c=1$ gives the Genocchi numbers (see \eqref{genocchi}). 
In other words, the Genocchi numbers $G_{2n+2}$ are the moments of
the continuous dual Hahn polynomials $p_n(x, 0,1,1)$.

As we recall, Carlitz~\cite{Ca2} gave a cumbersome formula 
for the
exponential generating function of $\mu_n(a,b,c)$. By \eqref{gfwzero} the corresponding generating functions for $a=0$ is
 \begin{equation}\label{mugfzero}
\sum_{n=0}^\infty \mu_{n}(0, b,c)\frac{t^{2n}}{(2n)!}
={}_2F_1\left(\begin{array}{c}
b, \,c\cr
\half
\end{array};\sin^2\frac{t}{2}\right),
\end{equation}
which is equivalent to Carlitz's \cite[Equation (4.2)]{Ca2}.
In particular, the exponential generating function of Genocchi numbers $G_{2n+2}=\mu_n(0,1,1)$ also has the
hypergeometric series representations:
\begin{align*}
1+\sum_{n=1}^{\infty}G_{2n+2}{t^{2n}\over (2n)!}
&={}_2F_1\left(
\ab{\,1,\ 1}{\half}
;\sin^2\frac t2  \right)\\
&=\sec^2\frac t2 \,{}_2F_1\left(
\ab{\,1, \ -\half}{\half}
;-\tan^2\frac t2\right),
\end{align*}
where the last formula follows from Pfaff's transformation (see \cite[p.~68]{AAR}).

The J-fraction for $\mu_n(\tfrac12,b,c)$ is 
\begin{multline}
\sum_{n=0}^\infty \mu_{n}(\tfrac12,b,c)t^n=J(t;\; bc+\frac{1}{2}(b+c),\; (b+c)(b+1/2)(c+1/2),\\ 
 \qquad bc+\frac{5}{2}(b+c)+2\cdot 1^2,\;(b+c+1)(b+3/2)(c+3/2),\\
  bc+\frac{9}{2}(b+c)+2\cdot 2^2, (b+c+2)(b+5/2)(c+5/2), \ldots). 
\end{multline}

By \eqref{gfwhalf} the corresponding generating functions for 
 $a=\half$ is
 \begin{equation}\label{mugfhalf}
 \sum_{n=0}^\infty \mu_{n}(\half, b,c)\frac{t^{2n+1}}{(2n+1)!}
=2\sin\frac{t}{2} \;{}_2F_1\left(\begin{array}{c}
b+\half, \,c+\half\cr
\tfrac {3}{2}
\end{array};\sin^2\frac{t}{2}\right).
\end{equation}
The following  two nice special cases of \eqref{mugfzero} and \eqref{mugfhalf} are known (see \cite[p.~101]{Er}):
\begin{align}
\frac{\cos (at/2)}{\cos t/2}&={}_2F_1\left(\begin{array}{c}
(1-a)/2, \,(1+a)/2\cr
1/2
\end{array};\sin^2\frac{t}{2}\right),\\
\frac{\sin(at/2)}{a\sin t/2}&={}_2F_1\left(\begin{array}{c}
(1-a)/2, \,(1+a)/2\cr
3/2
\end{array};\sin^2\frac{t}{2}\right).
\end{align}

To find the orthogonal polynomials whose moments are median Genocchi numbers we shall consider the \emph{generalized 
Dumont-Foata polynomials} in 6 variables due to Dumont~\cite{Du}:
$$
\Gamma_{n+1}(\alpha, \bar \alpha):=
\Gamma_{n+1}(\alpha, \bar \alpha, \beta, \bar\beta, \gamma, \bar\gamma),
$$ 
which can be   defined by the
J-fraction 
\begin{multline}\label{dumont-ex-conj}
\sum_{n\geq 0}\Gamma_{n+1}(\alpha, \bar \alpha)t^n
=J(t;\,
\alpha\bar \beta+\beta\bar \gamma+\gamma\bar \alpha,\,
(\bar \alpha+\beta)(\bar \beta+\gamma)(\bar \gamma+\alpha),\,\ldots,  \\
(\alpha+n)(\bar \beta+n)+(\beta+n)(\bar \gamma+n)+(\gamma+n)(\bar \alpha+n) -n(n+1),\\
 (n+1)(\bar \alpha+\beta+n)(\bar \beta+\gamma+n)(\bar \gamma+\alpha+n),\ldots). 
\end{multline}

\begin{thm}
The sequence $\{\Gamma_{n+1}(\alpha,\bar \alpha)\}_n$ is the moment sequence of the rescaled 
continuous dual Hahn polynomials $Z_n(x)=p_n(x+d;a,b,c)$ with 
\begin{equation}
\label{abcd}
\left\{
\begin{aligned}
a&=\half(\alpha +\bar \alpha +\beta -\bar \beta +\bar \gamma -\gamma), \\
b&=\half(\bar \alpha-\alpha+\beta+\bar \beta+\gamma-\bar \gamma),\\
c&=\half(\alpha-\bar \alpha-\beta+\bar \beta+\gamma+\bar \gamma),\\
d&=\alpha\bar \alpha+\alpha(\beta-\bar \beta)-\bar \alpha(\gamma-\bar \gamma)-a^2.
\end{aligned}
\right.
\end{equation}
More precisely, let  $\psi \colon K[x]\longrightarrow K$  be a {\sl linear functional}
such that $\psi(x^n)=\Gamma_{n+1}(\alpha,\bar \alpha)$ for $n\geq 0$. Then
\begin{equation}\label{zorth}
\psi(Z_m(x)Z_n(x))=n!\,(\bar \alpha+\beta)_n(\alpha+\bar \gamma)_n
(\bar \beta+\gamma)_n\delta_{mn}.
\end{equation}
where
\begin{align}\label{defZ}
Z_n(x)
&=\sum_{k=0}^n(-1)^{n-k}{n\choose k}(\bar a+\beta +k)_{n-k}
(\alpha+\bar \gamma +k)_{n-k}\nonumber \\
&\quad\times \prod_{l=0}^{k-1}\left[x^2+(\alpha+l)
(\bar \alpha +l)+(\alpha +l)(\beta -\bar \beta )-(\bar \alpha+l)(\gamma-\bar \gamma)\right].
\end{align}
\end{thm}
\begin{proof} 
It is known that the continuous dual Hahn polynomials $p_n(x)$ have  the recurrence relation~\cite[(9.3.5)]{KLS}:
\begin{equation}\label{recurrence-dual Hahn}
p_{n+1}(x)=(x-b_n)p_n(x)-\lambda_{n}p_{n-1}(x) \qquad (n\geq 0),
\end{equation}
where
\begin{align}\begin{cases}
b_n&=(n+a)(n+b)+(n+a)(n+c)+(n+b)(n+c)-n(n+1),\\
\lambda_{n}&=n(n-1+a+b)(n-1+a+c)(n-1+b+c).
\end{cases}
\end{align}
Thus, the polynomials  $Z(x)=p_n(x+d, a,b,c)$ with the substitution
\eqref{abcd}
have 
the recurrence relation
\begin{equation}\label{re3}
Z_{n+1}(x)=(x-\bar{b}_n)Z_n(x)-\bar{\lambda}_nZ_{n-1}(x) \qquad (n\geq 0)
\end{equation}
with
\begin{equation}\label{bn}
\begin{cases}
\bar{b}_n=(n+\alpha)(n+\bar \beta)+(n+\bar \alpha)(n+\gamma)
+(n+\beta)(n+\bar \gamma)-n(n+1),\\
\bar{\lambda}_{n}=n(n-1+\bar \alpha+\beta)
(n-1+\alpha+\bar \gamma)(n-1+\bar \beta+\gamma).
\end{cases}
\end{equation} 
This recurrence is equivalent 
to the J-fraction \eqref{dumont-ex-conj} for  $\Gamma_n(\alpha,\bar \alpha)$.
\end{proof}

\begin{prop} We have 
\begin{align}\label{recurrence-gamma}
\Gamma_{n+1}(\alpha,\bar \alpha)&=(\alpha+\bar \gamma)(\beta+\bar \alpha)
\Gamma_{n} (\alpha+1, \bar \alpha+1)\nonumber\\
 &\qquad+ [\alpha(\bar \beta-\beta)-\bar \alpha(\bar \gamma-\gamma)-\alpha
\bar \alpha]\Gamma_{n}(\alpha,\bar \alpha)
\end{align}
with $\Gamma_1(\alpha,\bar \alpha) =1$,
and 
\begin{align}\label{ogfgDF}
\sum_{n\geq 0}& \Gamma_{n+1}(\alpha, \bar \alpha)t^n\nonumber\\
&=\sum_{n\geq 0}\frac{(\alpha+\bar\gamma)_{n}(\beta+\bar \alpha)_{n}t^n}{\prod_{k=0}^{n}
\left( 1-[(\alpha+k)(\bar\beta-\beta)-(\bar\alpha+k)(\bar\gamma-\gamma)-(\alpha+k)(\bar\alpha+k)]t \right)}.
\end{align}
\end{prop}
\begin{proof}
Clearly, the recurrence \eqref{recurrence-gamma} is equivalent to the generating function~\eqref{ogfgDF}. So we just prove \eqref{ogfgDF}.
It is known (see \cite{Ch}) and easy to see  that 
the moments of 
 $p_n(x+d; a,b,c)$ are related to those of $p_n(x;a,b,c)$ as follows:
\begin{align}
\Gamma_{n+1}(\alpha, \bar \alpha)=\sum_{k=0}^n{n\choose k} (-d)^{n-k} \mu_k(a,b,c).
\end{align}
Therefore 
\begin{align*}
\sum_{n\geq 0} \Gamma_{n+1}(\alpha, \bar \alpha)t^n&=
\sum_{n\geq 0}\sum_{k=0}^n{n\choose k} (-d)^{n-k} \mu_k(a,b,c)\,t^n\\
&=\sum_{k\geq 0} \mu_k(a,b,c)t^k \,\sum_{n\geq 0}{n+k\choose k}(-d\,t)^{n}\\
&=\sum_{k\geq 0} \mu_k(a,b,c)t^k(1+dt)^{-k-1}.
\end{align*}
Invoking \eqref{gf-F} and \eqref{abcd} we 
derive the generating function.
\end{proof}

\begin{remark}
Originally 
Dumont \cite{Du} defined the polynomials 
$\Gamma_n(\alpha,\bar \alpha)$ combinatorially and 
conjectured the J-fraction in \eqref{dumont-ex-conj}. 
Randrianarivony~\cite{Ra}
and Zeng~\cite{Ze} proved Dumont's conjectured J-fraction by first establishing \eqref{ogfgDF} and \eqref{recurrence-gamma}
 from Dumont's combinatorial definition for 
 $\Gamma_{n+1}(\alpha,\bar \alpha)$. In 2010
Josuat-Verg\`es \cite{JV} gave a new 
 proof of the  J-fraction\eqref{dumont-ex-conj} starting from \eqref{recurrence-gamma}.
\end{remark}

\begin{cor}
The median Genocchi numbers  
$H_{2n+1}$ are the moments of the rescaled continuous dual Hahn polynomials 
$p_n\bigl(x-\frac{1}{4}, \frac{1}{2}, \frac{1}{2},  \frac{1}{2}\bigr)$.
\end{cor}
\begin{proof}
Contracting the S-fraction \eqref{mgenocchi} 
by  \eqref{contraction1}
 we obtain  the following J-fraction
\begin{align}\label{J-frac-MG}
\sum_{n=0}^{\infty}H_{2n+1}t^{n}=J(t; 2\cdot 1^2, (1\cdot 2)^2,2\cdot 2^2, (2\cdot 3)^2, \ldots, 2\cdot n^2, (n\cdot (n+1))^2, \ldots). 
\end{align}
Comparing with \eqref{dumont-ex-conj} we see that 
$H_{2n+1}=\Gamma_{n+1}(1,1,1, 0,1,1)$. Hence $a=b=c=\frac{1}{2}$ and 
$d=-\frac{1}{4}$ by \eqref{abcd}.
\end{proof}

Our method does not  produce  an exponential generating function for the median Genocchi numbers here since the denominator in \eqref{ogfgDF} with $\alpha=1$ and $\bar \alpha=0$ does not  factorize nicely.
\section{Moments of Hahn polynomials}

In this section we introduce another method for determining  generating functions for moments of orthogonal polynomials. We apply it only to the Hahn polynomials, but it can also be used for the Wilson polynomials and 
Askey-Wilson polynomials (see  Theorem \ref{AWmoment}).

\begin{lem}
\label{l-nicole}
Let $a_0, a_1, a_2,\ldots$ be arbitrary. Then
\begin{equation}\label{nicole}
\sum_{n=0}^\infty (x+a_0)(x+a_1)\cdots
(x+a_{n-1})\frac{t^n}{(1+a_0t)\cdots (1+a_nt)}=\frac{1}{1-xt}.
\end{equation}
\end{lem}
\begin{proof}
We have the indefinite sum
\begin{multline}\
\qquad\sum_{n=0}^m(x+a_0)\cdots (x+a_{n-1})\frac{t^n}{(1+a_0t)\cdots
(1+a_nt)}\\
=\frac{1}{1-xt}\left[1-(x+a_0)\cdots (x+a_m)\frac{t^{m+1}}{(1+a_0t)\cdots
(1+a_mt)}\right],\qquad
\end{multline}
which is easily
proved by induction. The lemma follows by taking $m\to \infty$.
\end{proof}

\begin{remark} As pointed out by Knuth \cite[equation (2.16)]{Kn}, a formula equivalent to Lemma \ref{l-nicole} was
discovered by Fran\c cois  Nicole \cite{Ni} in 1727. Nicole's formula was also used by Ap\'ery to derive the formula 
\[\zeta(3) = \frac52\sum_{n=1}^\infty \frac{(-1)^{n-1}}{n^3\binom{2n}{n}};\]
see van der Poorten \cite[section 3]{VDP}.
The identity in the proof is equivalent to a special 
case of Newton's interpolation formula \textup{(}see \cite{MT}\textup{)}. Moreover,
equating  coefficients of $t^n$ in \eqref{nicole} gives
the expansion of 
the polynomial $x^n$ in the basis
  $\{(x+a_0)\cdots (x+a_{k})\}_{0\leq k\leq n-1}$, in which the coefficents  can be computed 
   by Newton's interpolation formula. We give 
    such an example for the moments of Askey-Wilson polynomials  at the end of this paper. 
\end{remark}

The following result is an immediate consequence of Lemma \ref{l-nicole}.

\begin{prop}
\label{p-moments}
Let $L$ be a linear functional defined on polynomials in $x$ and let $a_0,a_1,\dots$ be complex numbers or indeterminates that do not involve $x$. Let 
\begin{equation*}
v_n = L\bigl((x+a_0)(x+a_1)\cdots(x+a_{n-1})\bigr).
\end{equation*}
Then 
\begin{equation}\label{momentgf}
\sum_{n=0}^\infty L(x^n) t^n = \sum_{n=0}^\infty v_n\frac{t^n}{(1+a_0t)\cdots (1+a_nt)}.
\end{equation}
\end{prop}
\begin{remark}\label{rem5}  
Substituting $x$ by $x+T$ and 
$a_j$ by $a_j-T$ ($j\in \N$) for any complex variable  $T$ in the above equation yields the 
  formula for  $L((x+T)^n)$.
\end{remark}

The Hahn polynomials are defined in \cite{KLS} by
\begin{equation} 
\label{hahn1}
Q_n(x; \alpha, \beta, N)={}_3F_2\left(\begin{array}{c} -n, n+\alpha+\beta+1, -x\cr
\alpha+1, -N
\end{array};1\right).
\end{equation}
(Some authors define $Q_n(x; \alpha, \beta, N)$ with $-N+1$ replacing $-N$ in \eqref{hahn1}.)
If $N$ is a nonnegative integer then the polynomials  $Q_n(x; \alpha, \beta, N)$  are defined 
only for $n=0,1,\dots, N$ and it is known that they
are orthogonal with respect to the linear functional $L_0$ given by
\begin{equation}
\label{e-L0}
L_0(p(x))=\sum_{x=0}^N\binom{\alpha+x}{x}\binom{\beta+N-x}{N-x}p(x).
\end{equation}
Applying Vandermonde's theorem we find that
$$
L_0\bigl((x+\alpha+1)_n\bigr)=\frac{(\alpha+\beta+2)_N}{N!}\cdot\frac{(\alpha+1)_n(\alpha+\beta+N+2)_n}{(\alpha+\beta+2)_n}.
$$

This suggests that in studying moments of Hahn polynomials we reparametrize them  by setting $\alpha=A-1$, 
$\beta=C-A-1$,  and $N=B-C$ so that $\alpha+1=A$, $\alpha+\beta+N+2=B$, and $\alpha+\beta+2=C$.


Thus we  define polynomials $R_n(x;A,B,C)$ by 
\begin{align*}
R_n(x; A, B, C)&=Q_n(x; A-1,C-A-1, B-C)\\
&={}_3F_2\left(\begin{array}{c}
-n, n+C-1, -x\cr
A, C-B
\end{array};1\right)
\end{align*}
where $A$, $B$, and $C$ are indeterminates. We will show that these polynomials are orthogonal with respect to the  linear functional $L$ on polynomials in $x$ defined by
\begin{equation}
\label{e-HahnL}
L\bigl((x+A)_m\bigr)=\frac{(A)_m(B)_m}{(C)_m}.
\end{equation}

The polynomials $R_n(x;A,B,C)$ are closely related to the continuous Hahn polynomials,  defined in \cite{KLS} by 
\begin{equation*}
p_n(x;a,b,c,d) = i^n \frac{(a+c)_n(a+d)_n}{n!} {}_3F_2\left(\begin{array}{c}
-n, n+a+b+c+d-1, a+ix\cr
a+c,a+d
\end{array};1\right).
\end{equation*}
Thus
\begin{equation*}
R_n(x;A,B,C)= (-i)^n \frac{n!} {(A)_n (C-B)_n} p_n(ix; 0, B-A, A,C-B)
\end{equation*}
and
\begin{equation}
p_n(x;a,b,c,d) =  i^n \frac{(a+c)_n(a+d)_n}{n!} R_n(-a-ix; a+c, b+c, a+b+c+d).
\end{equation}

\begin{lem}
For nonnegative integers $m$ and $n$ we have
\begin{equation}
\label{e-HahnL1}
L\bigl((x+A)_m(-x)_n\bigr)=\frac{(A)_{m+n}(B)_m(C-B)_n}{(C)_{m+n}}.
\end{equation}
\end{lem}
\begin{proof}
By Vandermonde's theorem we have 
\begin{equation*}
(-x)_n=\sum_{i=0}^n (-1)^i \binom{n}{i} (x+A+m)_i (A+m+i)_{n-i}.
\end{equation*}
Thus 
\begin{align*}
L\bigl((x+A)_m(-x)_n\bigr) &= L\biggl( \sum_{i=0}^n (-1)^i \binom{n}{i} (x+A)_{m+i} (A+m+i)_{n-i} \biggr)\\
  &=\sum_{i=0}^n (-1)^i \binom{n}{i} \frac{(A)_{m+i}(B)_{m+i}}{(C)_{m+i}} (A+m+i)_{n-i}\\
  &=\frac{(A)_{m+n}(B)_m(C-B)_n}{(C)_{m+n}}
\end{align*}
by Vandermonde's theorem again.
\end{proof}

%

\begin{prop}
The polynomials $R_n(x;A,B,C)$ are orthogonal with respect to the linear functional $L$, and 
\begin{equation}
\label{e-RR}
L\bigl(R_n(x; A,B,C)^2\bigr) = (-1)^n n!\frac{(B)_n(C-A)_n}{(A)_n (C-B)_n (C)_{n-1}(C+2n-1)}.
\end{equation}
\end{prop}
\begin{proof}
To show that the polynomials $R_n$ are orthogonal, it suffices to show that for $0\le m<n$,
$$
L(R_n(x; A,B,C)(x+A)_m)=0.
$$
We have
\begin{align}
L(R_n(x; A,B,C)(x+A)_m)&= L\left(\sum_{k=0}^n\frac{(-n)_k(n+C-1)_k}{k!\,(A)_k(C-B)_k}
    (-x)_k(x+A)_m\right)\notag\\ 
&= \sum_{k=0}^n\frac{(-n)_k(n+C-1)_k}{k!\,(A)_k(C-B)_k} \notag
\frac{(A)_{m+k}(B)_m(C-B)_k}{(C)_{m+k}}\notag\\
&= \frac{(A)_{m}(B)_m}{(C)_m}{}_3F_2\left(\begin{array}{c}
-n, n+C-1, A+m\cr
A, C+m 
\end{array};1\right)\notag\\
&= \frac{(A)_{m}(B)_m}{(C)_m}\frac{(-m)_n(A-n-C+1)_n}{(A)_n(-m-n-C+1)_n}\notag\\
&= \frac{(A)_{m}(B)_m(C-A)_n(-m)_n}{(C)_{m+n}(A)_n}, \label{e-L3}
\end{align}
where we have used the Pfaff-Saalsch\"utz theorem to evaluate the $_3F_2$.
So if $0\le m<n$ this is~0. 
Since $R_n(x; A,B,C)$ has leading coefficient $(C-1)_{2n}/(A)_n(C-B)_n (C-1)_n$, the case $m=n$ of \eqref{e-L3} gives
\eqref{e-RR}.
\end{proof}

\begin{prop}
\label{p-Mgf}
Let $M_n(A,B,C)= L(x^n)$. Then 
\begin{equation}\label{Mgf}
\sum_{n=0}^\infty M_n(A,B,C)t^n=
\sum_{n=0}^\infty\frac{(A)_n(B)_n}{(C)_n}\frac{t^n}{\prod_{l=0}^n(1+(A+l)t)},
\end{equation}
and 
\begin{equation}\label{Megf}
\sum_{n=0}^\infty M_n(A,B,C)\frac{t^n}{n!}=e^{-At}{}_2F_1
\left(\begin{array}{c}
A, \, B\cr
C
\end{array}; 1-e^{-t}\right).
\end{equation}
\end{prop}

\begin{proof}
Equation \eqref{Mgf} follows from 
 Proposition \ref{p-moments} and equation \eqref{e-HahnL}.
By Lemma \ref{l-egf} we have 
\begin{equation*}
\varepsilon\left(\frac{t^n}{\prod_{l=0}^n(1+(A+l)t)}\right) = e^{-(A+n)t}\frac{(e^t-1)^n}{n!}
  =e^{-At}\frac{(1-e^{-t})^n}{n!}
\end{equation*}
so applying $\varepsilon$ to \eqref{Mgf} gives \eqref{Megf}.
\end{proof}

A formula equivalent to \eqref{Megf} has been given by Dominici \cite[Section 3.7]{Dom}.

\begin{thm}\label{thm-S} The following S-fraction holds:
\begin{multline}
\sum_{n=0}^\infty M_n(A,B,C)t^n= S\left(t; \frac{A(B-C)}{C},\; \frac{B(C-A)}{C(C+1)},\right.\\ 
\frac{(A+1)(B-C-1)C}{(C+1)(C+2)},\; \frac{2(B+1)(C-A+1)}{(C+2)(C+3)},\\ 
\left. \frac{(A+2)(B-C-2)(C+1)}{(C+3)(C+4)},\;
\frac{3(B+2)(C-A+2)}{(C+4)(C+5)},\ldots\right).
\end{multline}
\end{thm}
\begin{proof}  Let $F(t; A, B, C)=\sum_{n=0}^\infty M_n(A,B,C) t^n$.
By \eqref{Mgf} we have 
\begin{subequations}
\begin{equation}\label{link:F and G}
F(t; A, B, C)=
\frac{1}{1+At} G\biggl(\frac{t}{1+At}; A,B,C\biggr)
\end{equation}
with
\begin{equation}\label{def:G}
G(t; A,B,C)=\sum_{n=0}^\infty\frac{(A)_n(B)_n}{(C)_n}\frac{t^n}{\prod_{l=0}^n(1+l\,t)}.
\end{equation}
\end{subequations}
We first expand $G(t; A,B,C)$ as an S-fraction
$$
G(t; A,B,C)=S(t; C_1, C_2, \ldots, C_n,\ldots),
$$
which can be written as a J-fraction 
\begin{equation}\label{J1}
G(t; A,B,C)= J(t; C_1, C_1C_2, C_2+C_3, C_3C_4, \ldots),
\end{equation}
where $C_n:=C_n(A,B,C)$ ($n\geq 1$) are to be determined. 

Rewriting \eqref{def:G} as
$$
G(t; A,B,C)=1+\frac{AB}{C}\cdot \frac{t}{1+t}\cdot G\biggl(\frac{t}{1+t}; A+1,B+1,C+1\biggr)
$$
and using the J-fraction \eqref{J1} to write
\begin{multline}
 \frac{1}{1+t}\cdot G\biggl(\frac{t}{1+t}; 
 A+1,B+1,C+1\biggr)\\
 = 
 J(t; C^+_1-1, C^+_1\,C^+_2, C^+_2+C^+_3-1, C^+_3\,C^+_4, \ldots)\nonumber
\end{multline}
 with $C^+_n=C_n(A+1, B+1, C+1)$, we obtain 
\begin{equation}\label{G-equation}
G(t; A,B,C)=1+\frac{AB}{C}t\cdot J(t; C^+_1-1, C^+_1\,C^+_2, C^+_2+C^+_3-1, C^+_3\,C^+_4, \ldots).
\end{equation}
On the other hand, 
contracting the S-fraction by
\eqref{contraction2}  we have 
\begin{equation}\label{J2}
G(t; A,B,C)=1+C_1t\cdot J(t; C_1+C_2, C_2C_3, C_3+C_4, C_4C_5, \ldots).
\end{equation}
Comparing the above two J-fractions 
we derive
\begin{align*}
 C_1=\frac{AB}{C},\quad C_1+C_2=C^+_1-1,\quad
 C_2C_3=C^+_1\,C^+_2,
 \end{align*}
 and  for $n\geq 2$ 
\begin{align*}
 C_{2n-1}+C_{2n}&=C^+_{2n-2}+C^+_{2n-3}-1,\\
 C_{2n}\,C_{2n+1}&=C^+_{2n-1}\,C^+_{2n+2}.
 \end{align*}
This yields immediately  
\begin{equation}\label{Cn}
\begin{cases}\displaystyle
C_{2n-1}=\frac{(A+n-1)(B+n-1)(C+n-2)}{(C+2n-3)(C+2n-2)};\\
\displaystyle
C_{2n}=\frac{n(B-C-n+1)(A-C-n+1)}{(C+2n-2)(C+2n-1)}.
\end{cases}
\end{equation}
Next,  by \eqref{link:F and G} and \eqref{J1} we have the J-fraction for $F(t;A,B,C)$:
\begin{subequations}
\begin{equation}\label{JF2}
F(t; A, B, C)=J(t; C_1-A, C_1C_2, C_2+C_3-A, C_3C_4, \ldots).
\end{equation}
It remains to determine  a sequence $\alpha_n$  such that 
\begin{align}
F(t; A, B, C)&=S(t; \alpha_1, \alpha_2, \ldots, \alpha_n,\ldots)\label{SF}\\
&=J(t; \alpha_1, \alpha_1\alpha_2, \alpha_2+\alpha_3, \alpha_3\alpha_4, \ldots). \label{JF1}
\end{align}
\end{subequations}
From \eqref{Cn}, \eqref{JF1} and \eqref{JF2} we derive 
\begin{align*}
\alpha_1&=C_1-A=\frac{A(B-C)}{C}\\
\alpha_2&=\frac{C_1C_2}{\alpha_1}=\frac{B(C-A)}{C(C+1)}.
\end{align*}
and for $n\geq 2$,
\begin{align*}
\alpha_{2n-1}&=C_{2n-1}+C_{2n-2}-A\\
&=\frac{(A+n)(B-C-n)(C+n-1)}{(C+n)(C+n+1)},\\
\alpha_{2n}&=\frac{C_{2n}C_{2n-1}}{\alpha_{2n-1}}\\
&=
\frac{n (B+n-1)(C-A+n-1)}{(C+2n-1)(C+2n-2)}.
\end{align*}
Plugging these values in  \eqref{SF} yields  the S-fraction in Theorem~\ref{thm-S}. 
\end{proof}
%

As an example of Proposition \ref{p-Mgf} let us consider the case $A=1,B=1, C=2$.
We have $_2F_1(1,1;2;z) = -z^{-1}\log(1-z)$, so 
\begin{equation*}
e^{-t}{}_2F_1
\left(\begin{array}{c}
1, \,1\cr
2
\end{array}; 1-e^{-t}\right)=\frac{t}{e^t-1}=\sum_{n=0}^\infty B_n \frac{t^n}{n!},
\end{equation*}
where $B_n$ is the $n$th  Bernoulli number. So $M_n(1,1,2)=B_n$.  

The orthogonal  polynomials whose moments are Bernoulli numbers were  considered by Touchard \cite{To}, but he was not able to find an explicit formula for them. An explicit formula (different from ours, but equivalent) was found by Wyman and Moser \cite{WM} and these polynomials were further studied by Carlitz \cite{Ca3}. 
One can show similarly that $M_n(1,2,3)=-2B_{n+1}$ and $M_n(1,1,3)=2(B_n+B_{n+1})$.
Krattenthaler  \cite[Section 2.7]{Kr} proved a result equivalent to 
\begin{equation*}
\frac{e^t}{6}\sum_{n=0}^\infty M_n(2,2,4)\frac{t^n}{n!} = \sum_{n=0}^\infty B_{n+2}\frac{t^n}{n!}.
\end{equation*}
Fulmek and Krattenthaler \cite[equation (5.14)]{FK} expressed the moments of general continuous Hahn polynomials (with some integrality and nonnegativity restrictions on the parameters) in terms of Bernoulli numbers, generalizing all of these formulas. 

Chapoton \cite{Cha} studied some special cases of Racah polynomials whose moments are 
the \emph{median Bernoulli numbers} since they are, up to an easy power of 2, 
the main diagonal of Seidel's difference tableau of Bernoulli numbers \cite[p.~181]{Sei}.

For the moments $\mu_n$ of the original Hahn polynomials \[Q_n(x;\alpha, \beta, N) = R_n(x; \alpha+1, \alpha+\beta+2, \alpha+\beta+N+2),\] equation \eqref{Megf} gives
\begin{equation*}
\sum_{n=0}^\infty \mu_n \frac{t^n}{n!} = 
e^{-(\alpha+1)t}{}_2F_1
\left(\begin{array}{c}
\alpha+1, \, \alpha+\beta+N+2\cr
\alpha+\beta+2
\end{array}; 1-e^{-t}\right).
\end{equation*}
Applying Pfaff's transformation gives
\begin{equation*}
\sum_{n=0}^\infty \mu_n \frac{t^n}{n!} = 
{}_2F_1
\left(\begin{array}{c}
\alpha+1, \, -N\cr
\alpha+\beta+2
\end{array}; 1-e^{t}\right).
\end{equation*}
If $N$ is a nonnegative integer, we may apply the terminating $_2F_1$ transformation
\begin{equation*}
{}_2F_1
\left(\begin{array}{c}
a, \, -N\cr
b
\end{array};z\right)
=\frac{(b-a)_N}{(b)_N}
{}_2F_1
\left(\begin{array}{c}
a, \, -N\cr
1+a-b-N
\end{array};1-z\right)
\end{equation*}
to obtain
\begin{equation*}
\sum_{n=0}^\infty \mu_n \frac{t^n}{n!} = 
\frac{(\beta+1)_N}{(\alpha+\beta+2)_N}
{}_2F_1
\left(\begin{array}{c}
\alpha+1, \, -N\cr
-\beta-N
\end{array};e^t\right).
\end{equation*}
Equating coefficients of $x^n/n!$ gives
\begin{align*}
\mu_n &= \frac{(\beta+1)_N}{(\alpha+\beta+2)_N}\sum_{k=0}^N\frac{(\alpha+1)_k (-N)_k}{k!\,(-\beta-N)_k}k^n\\
  &= \frac{N!}{(\alpha+\beta+2)_N}\sum_{k=0}^n \binom{\alpha+k}{k}\binom{\beta+N-k}{N-k}k^n,
\end{align*}
so we see that, as expected, in this case the linear functional $L$ is a constant multiple of the linear functional $L_0$ defined by 
\eqref{e-L0}.

\section{Concluding remarks}
\label{s-concluding}
In what follows the standard  $q$-notations (see \cite{AAR, Is,KLS}) will be used.
The Askey-Wilson polynomials are defined by 
$$
p_n(x;a,b,c,d\,|\, q):=(ab,ac,ad;q)_{n} a^{-n}A_{n}(x)\qquad (n\in \N)
$$ 
with
\begin{align}
A_{n}(x)=
{}_{4}\phi_{3}\left[\!\!\! \begin{array}{c}
q^{-n},\, abcdq^{n-1},\,ae^{i\theta},\, ae^{-i\theta}\\
ab,\, ac,\, ad\end{array}\!\!\!;q,q\right],
\end{align}
where $x=\cos\theta$, see \cite{AW, Is, KLS}.

In the last decade much work has been done to extend Viennot's results for 
moments of classical orthogonal polynomials 
to   the moments of the Askey-Wilson polynomials; see \cite{CKS, CSSW, KS, GITZ}. It would be interesting to see to what extent the methods of this paper  can be $q$-generalized.
We show that 
 Proposition~\ref{p-moments} works also  for the moments of the Askey-Wilson polynomials. 

For $0<q<1$, $\max\{|a|, |b|, |c|, |d|\}<1$,  $z=e^{i\theta}$ and $x=\cos\theta$,
the linear functional  $\L_q: \C[x]\mapsto \C$  associated to the orthogonal  measure of  the Askey-Wilson polynomials
has the explicit integral form \cite{AW}:
\begin{multline}
\L_q\bigl(x^n\bigr)=\frac{1}{2\pi}
\frac{(ab,ac, ad, bc, bd, cd;q)_\infty}{(abcd;q)_\infty}\\
\times \int_0^\pi 
 \frac{{(\cos \theta)^n}\;(e^{2 i\theta}, e^{-2 i\theta};q)_\infty \;d\theta}
 {(ae^{i\theta}, ae^{-i\theta},be^{i\theta}, be^{-i\theta},ce^{i\theta}, ce^{-i\theta},de^{i\theta}, de^{-i\theta};q)_\infty}.\qquad\nonumber
\end{multline} 
Thus, the Askey-Wilson integral reads 
 $\L_q(1)=1$. As
 $(az,a/z;q)_n=(ae^{i\theta}, ae^{-i\theta}; q)_n$, the value 
 $\L_q\bigl((az,a/z;q)_n\bigr)$ amounts to shifting $a$ to $aq^{n}$ in the integral, so
 \begin{align}\label{linfunc}
\L_q\bigl((az,a/z;q)_n\bigr)
&=\frac{(ab,ac,ad;q)_n}{(abcd;q)_n}.
\end{align}
The same argument 
yields  the $q$-version of \eqref{mnMoments}:
 \begin{align}\label{mnMomentsAW}
\L_q\bigl((az,a/z;q)_n(bz,b/z;q)_m\bigr)
&=\frac{(ab,q)_{m+n} (ac,ad;q)_n(bc, bd:q)_m}{(abcd;q)_{m+n}}.
\end{align}

As   $x=(z+1/z)/2$, choose 
$a_j=(q^{-j}/a+aq^j)/2$ for $j\in \N$ and $t\in \C$. Then  
$$
(x-a_0)\cdots (x-a_{n-1})=(-1)^n(2a)^{-n}q^{-{n\choose 2}}(az, a/z;q)_n.
$$
Applying  Proposition~\ref{p-moments} to \eqref{linfunc}
we derive  the generating function 
of the moments of Askey-Wilson polynomials. 
\begin{thm}\label{AWmoment} We have 
\begin{equation}\label{eq:AWmoment}
\sum_{n=0}^\infty \L_q\bigl(x^n\bigr) u^n=
\sum_{n=0}^\infty \frac{(ab,ac,ad;q)_n}{(abcd;q)_n}
\frac{(-1)^n (2a)^{-n}q^{-{n\choose 2}} u^n}{(1-a_0 u)\ldots (1-a_{n}u)}.
\end{equation}
\end{thm}
It is interesting to note that
computing the coefficent of $u^n$  in the right side of \eqref{eq:AWmoment}
by partial fraction decomposition  or using Newton's interpolation formula \cite{GITZ} yields
the $t=0$ case of the double sum formula for the Askey-Wilson moments in \cite[Theorem 1.13]{CSSW}. The  general case follows by applying the shifted version of 
Proposition~\ref{p-moments}, see Remark~\ref{rem5}.
Finally, we note that 
 an algebraic proof of the orthogonality relation of the Askey-Wilson polynomials was given by Gasper and Rahman  in \cite[p.~190--191]{GR} using \eqref{mnMomentsAW}.


\end{document}